\documentclass[11pt,twoside]{amsart}
\usepackage{amsxtra}
\usepackage{color}
\usepackage{amsmath,amsthm,amssymb}
\usepackage{mathrsfs,mathtools}
\usepackage{stmaryrd}
\usepackage{hyperref}
\usepackage{enumerate}
\usepackage{xcolor}
\usepackage{pifont}
\usepackage{adjustbox}
\usepackage{tikz}

\theoremstyle{theorem}
\newtheorem{theorem}{Theorem}[section]

\newtheorem{proposition}[theorem]{Proposition}

\newtheorem*{theorem*}{Theorem}
\theoremstyle{definition}

\newtheorem*{question}{Question}
\theoremstyle{remark}
\newtheorem{remark}[theorem]{Remark}

\def\sideremark#1{\ifvmode\leavevmode\fi\vadjust{
		\vbox to0pt{\hbox to 0pt{\hskip\hsize\hskip1em
				\vbox{\hsize3cm\tiny\raggedright\pretolerance10000
					\noindent #1\hfill}\hss}\vbox to8pt{\vfil}\vss}}}

\newcommand{\R}{{\mathbb R}}

\newcommand{\T}{{\mathbb T}}

\newcommand{\beq}{\begin{equation}}
\newcommand{\eeq}{\end{equation}}
\renewcommand{\a}{\alpha}
\renewcommand{\b}{\beta}

\renewcommand{\d}{\delta}

\newcommand{\g}{\gamma}

\renewcommand{\o}{\omega}

\newcommand{\U}{{\mathrm U}}
\newcommand{\SU}{{\mathrm{SU}}}

\newcommand{\SO}{{\mathrm {SO}}}

\newcommand{\M}{{\mathrm M}}
\renewcommand{\T}{{\mathrm T}}

\newcommand{\G}{{\mathrm G}}
\newcommand{\K}{{\mathrm K}}
\renewcommand{\L}{{\mathrm L}}
\newcommand{\Z}{{\mathrm Z}}

\newcommand{\W}{\wedge}

\DeclareMathOperator\ad{ad}

\newcommand{\Ric}{{\rm Ric}}

\newcommand{\Scal}{{\rm Scal}}

\renewcommand{\gg}{\mathfrak{g}}

\newcommand{\gk}{\mathfrak{k}}

\newcommand{\gm}{\mathfrak{m}}

\newcommand{\gu}{\mathfrak{u}}

\newcommand{\gz}{\mathfrak{z}}

\newcommand{\su}{\mathfrak{su}}

\newcommand{\st}{\ |\ }

\numberwithin{equation}{section}

\title[Bismut Ricci flat manifolds with symmetries]{Bismut Ricci flat manifolds with symmetries}
\author{Fabio Podest\`a and Alberto Raffero}
\address{Dipartimento di Matematica e Informatica ``U.~Dini'' \\ Universit\`a degli Studi di Firenze\\ Viale Morgagni 67/a\\ 50134 Firenze\\ Italy\\ }
\email{fabio.podesta@unifi.it}
\address{Dipartimento di Matematica ``G.~Peano'' \\ Universit\`a degli Studi di Torino\\
Via Carlo Alberto 10\\
10123 Torino\\ Italy}
\email{alberto.raffero@unito.it}

\begin{document}

\begin{abstract}
We construct examples of compact homogeneous Riemannian manifolds admitting an invariant Bismut connection that is Ricci flat and non-flat,
proving in this way that the generalized Alekseevsky-Kimelfeld theorem does not hold. 
The classification of compact homogeneous Bismut Ricci flat spaces in dimension $5$ is also provided.
Moreover, we investigate compact homogeneous spaces with non trivial third Betti number, and we point out other possible ways to construct Bismut Ricci flat manifolds.
Finally, since Bismut Ricci flat connections correspond to fixed points of the generalized Ricci flow, we discuss the stability of some of our examples under the flow.
\end{abstract}

\maketitle

\section{Introduction}
Let $(\M,g)$ be a Riemannian manifold, denote by $\nabla^g$ its Levi Civita connection, and consider a non-vanishing $3$-form $H\in\Omega^3(\M)$.
The {\em Bismut connection} associated with the pair $(g,H)$ is defined via the identity
\[
g(\nabla_XY,Z) = g(\nabla^g_XY,Z) + \frac12 H(X,Y,Z),
\]
for all $X,Y,Z\in\Gamma(T\M)$, it is the unique metric connection on $\M$ with totally skew-symmetric torsion $H$, and it has the same geodesics as $\nabla^g$.

These connections were successfully used in index theory problems in complex non-K\"ahler geometry \cite{Bis}, where they are characterized as the only complex metric connections with totally skew-symmetric torsion $H=d^c\o$
on a given Hermitian manifold $(\M,g,J)$, see e.g.~\cite{G}.
In this setting, {\em strong K\"ahler with torsion} (SKT) complex manifolds are precisely the Hermitian manifolds $(\M,g,J)$ whose K\"ahler form $\o$ satisfies the condition $dd^c\o=0$,
or, equally, whose Bismut connection has closed torsion, see for instance the survey \cite{FiTo} for more details.
On the other hand, the class of Riemannian manifolds admitting a Bismut connection $\nabla$ whose torsion $H$ is $\nabla$-parallel has been throughly studied in the literature, 
as this condition naturally holds for several geometrically significant structures as naturally reductive spaces, nearly K\"ahler and Sasakian structures among others, see e.g.~\cite{Agr,AFF,CMS}. 
Furthermore, Bismut connections are also of interest in theoretical and mathematical physics, see \cite{FI} and the references therein for a detailed explanation.

Bismut connections with closed torsion form play a central role in generalized Riemannian geometry, where they are naturally associated with generalized metrics on exact Courant algebroids, see \cite{Gar,FS}.
In this case, since the torsion of $\nabla$ is non-vanishing, the Ricci tensor $\Ric^{\nabla}$ is not symmetric, and one has (see \cite[Prop.~3.18]{FS})
\begin{equation}\label{RicB}
\Ric^\nabla = \Ric_g -\frac14 H^2 - \delta_g H,
\end{equation}
where $\Ric_g$ denotes the Ricci tensor of $\nabla^g$, $\delta_g$ is the formal adjoint of $d$, and the symmetric 2-tensor $H^2$ is defined as $H^2(X,Y) \coloneqq g(\imath_XH,\imath_YH)$, for every $X,Y\in\Gamma(T\M)$.

It is clear from \eqref{RicB} that a Bismut connection $\nabla$ with closed torsion form $H$ is {\em Ricci flat}, i.e., $\Ric^\nabla=0$,
if and only if $H$ is a $g$-harmonic $3$-form and the Ricci tensor of $g$ satisfies the equation $\Ric_g = \frac14 H^2$.
A pair $(g,H)$ giving rise to a Ricci flat Bismut connection $\nabla$ is called a {\em Bismut Ricci flat pair} ({\em BRF pair} for short) throughout the following.
In generalized Riemannian geometry, such pairs correspond to special types of generalized Einstein structures, see \cite[Ch.~3]{FS} for more information.
We recall here that for a BRF pair $(g,H)$ the scalar curvature $\Scal_g$ and the norm of $H,$ which are related by the identity $\Scal_g = \frac14 ||H||^2$,  are constant on $\M$, see  \cite[Lemma 2.24]{Lee}.

\smallskip

BRF pairs are fixed points of the {\em generalized Ricci flow}, a geometric flow introduced in \cite{CFMP,OSW}
in the context of renormalization group flows of two-dimensional nonlinear sigma models.
To describe this flow, consider a family of Riemannian metrics $g_t$ depending on a real parameter $t$, fix a closed $3$-form $H_0\in\Omega^3(\M)$ and let $H_t = H_0 + db_t$, where $b_t\in\Omega^2(\M)$.
Then, the generalized Ricci flow for $(g_t,b_t)$ is defined as follows
\begin{equation}\label{GRF}
\left\{
\begin{split}
\frac{\partial}{\partial t} g_t &= -2\Ric_{g_t} +\frac12 H^2_t,\\
\frac{\partial}{\partial t} b_t &= - \delta_{g_t} H_t,
\end{split}
\right.
\end{equation}
and it is well-posed on compact manifolds, see e.g.~\cite{FS}.
Notice that BRF pairs can also be regarded as trivial examples of {\em steady solitons} for the generalized Ricci flow. Indeed, the latter are defined by pairs $(g,H)$ satisfying the more general equations
\[
\Ric_g = \frac14 H^2 -\mathcal{L}_X g,\quad \delta_gH = -\imath_X H,
\]
for some vector field $X\in\Gamma(TM)$. The existence of non-trivial solitons on compact (complex) $4$-manifolds has been recently proved in \cite{Str4,StUs}.

Remarkably, the flow \eqref{GRF} can be seen as a generalization of Hamilton's Ricci flow to Bismut connections with closed torsion form \cite{Str} and as a flow of generalized metrics on exact Courant algebroids \cite{Gar,Str2}.
Moreover, it is related to some geometric flows in Hermitian Geometry, like the pluriclosed flow and the generalized K\"ahler Ricci flow, see e.g.~\cite{GJS,Str3,StTi0,StTi,StTi2}.
The reader may refer to the recent book \cite{FS} for an excellent introduction to the topic.

\smallskip

Standard examples of manifolds carrying BRF pairs are provided by compact simple Lie groups endowed with the bi-invariant metric (given by the negative of the Cartan-Killing form) and the standard harmonic $3$-form,
see e.g.~\cite[Prop.~3.53]{FS}. Notice that, in such a case, the corresponding Bismut connection is flat.
On the other hand, a simply connected compact Riemannian manifold admitting a flat Bismut connection is isometric to a product of compact simple Lie groups with bi-invariant metrics \cite[Thm.~3.54]{FS}.
It is currently not known whether other left-invariant BRF pairs may exist on compact Lie groups.

Since in the Riemannian case every homogeneous Ricci flat manifold is flat \cite{AK}, M.~Garcia-Fern\'andez and J.~Streets asked the following:
\begin{question}[\cite{FS}]
Given $(\M,g,H)$ a homogeneous Riemannian manifold with $H$ invariant and zero Bismut Ricci curvature, is the associated Bismut connection flat?
\end{question}
In this paper, we answer this question negatively.
After showing some general facts on compact homogeneous spaces with non-zero third Betti number in Section \ref{Sectb3}, we examine low dimensional compact homogeneous spaces in Section \ref{5dimsect}.
Since the $3$- and $4$-dimensional case are well understood from the results of \cite{FS}, we focus on $5$-dimensional compact homogeneous spaces 
and we obtain a full classification of those admitting invariant BRF pairs  in Theorem \ref{class5dim}. 
Beyond the case of compact Lie groups, we find a family of compact homogeneous spaces $\M_{p,q}$ parametrized by a pair of positive integers $p\geq q$ with $\mathrm{gcd}(p,q)=1$, 
all diffeomorphic to $S^3\times S^2$, where we prove the existence of invariant BRF pairs $(g,H)$
for which the corresponding Bismut connection $\nabla$ is not flat,  the metric $g$ is not Einstein and the torsion form $H$ is not $\nabla$-parallel, see Theorem \ref{BismutMpq}.
The uniqueness of such pairs is also studied in the same theorem.
Finally, in Section \ref{Mpqasymp}, we investigate the behavior of the homogeneous generalized Ricci flow on the spaces $\M_{p,q}$ with $p\neq q$,
showing that the invariant BRF pairs found in Theorem \ref{BismutMpq} are global attractors. 

As an additional remark, in Section \ref{KobaConst} we show that our examples are particular cases of a general construction by Kobayashi in the Riemannian Einstein setting \cite{K},
and we pave the way for a possible use of his construction to provide new examples of generalized Einstein manifolds, see Proposition \ref{KobaProp}.

\vspace{0.4cm}

{\bf Notation.} Throughout the paper, Lie groups will be denoted by capital letters and their  Lie algebras will be denoted by the respective gothic letters.
The Cartan-Killing form of a Lie algebra will be always denoted by $B$.
When a Lie group $\G$ acts on a manifold $\M$, the vector field associated to any $X\in\gg$ will be denoted by $\hat X$.
Finally, the space of Riemannian metrics on a manifold $\M$ will be denoted by $\mathcal{M}(\M)$.

\section{Compact homogeneous spaces with positive third Betti number}\label{Sectb3}
A preliminary step in the search for invariant Bismut Ricci flat connections on compact homogeneous spaces consists in finding conditions under which the third Betti number of the space is positive.

A typical example is given by compact semisimple Lie groups, where $b_3$ coincides with the number of simple factors, see \cite{ChEi}. In particular,
when $\G$ is a compact simple Lie group, then the third cohomology group $H^3(\G,\mathbb R)$ is generated by the {\em standard $3$-form} $\o$ defined as follows
\begin{equation}\label{stanform}
\omega(X,Y,Z) \coloneqq B([X,Y],Z),
\end{equation}
for every left-invariant vector fields $X,Y,Z$ on $\G$.
 More generally, in \cite{Aza} it was proved that the isotropy subgroup $\K$ of a compact homogeneous space $\M=\G/\K$ such that $b_3(\M) \geq 1$ must be finite whenever $\G$ is simple.
In the next theorem, we review this result and we obtain some new restrictions on $\K$ in the case where $\G$ is locally a product of two simple factors.

\begin{theorem}\label{b3}
Let $\G$ be a compact Lie group and let $\M=\G/\K$ be a $\G$-homogeneous space with $b_3(\M)\geq 1$. Then
\begin{enumerate}[a$)$]
\item\label{b3a} if $\G$ is simple, then $\K$ is finite;
\item\label{b3b} if $\G$ is locally a product of two simple factors $\G_1,\G_2$, then either the Lie algebra $\gk$ is contained in one of the two factors $\frak g_i$ or the Lie algebras $\frak g_i$ contain subalgebras $\gk_i$
isomorphic to $\gk$ and the projections $p_i:\gk\to\frak g_i$ are isomorphisms onto $\gk_i$, for $i=1,2$.
\end{enumerate}
\end{theorem}
\begin{proof}
The assertion \ref{b3a}) was proved in \cite{Aza}. We review the proof here.
Let $\pi:\G\to \G/\K$ be the projection and consider a closed $3$-form $\phi$ on $M$.
As $\G$ is compact, we can suppose that $\phi$ is $\G$-invariant.
If $\hat\phi\in \Omega^3(\G)$ is the closed $3$-form $\pi^*(\phi)$, then $\hat\phi$ is invariant under left $\G$-translations and right $\K$-translations.
Since $\G$ is simple, the third cohomology group $H^3(\G,\mathbb R)$ is generated by the standard $3$-form $\o$ defined in \eqref{stanform}.
Thus, $\hat\phi = c\,\omega + d\xi$, for some  $c\in\R$ and some $2$-form $\xi\in\Omega^2(\G)$.
Again by compactness, we can suppose that $\xi$ is invariant under left $\G$-translations and right $\K$-translations. This implies that, given $X,Y\in \frak g$ and $Z\in \gk$, we have
\[
\xi([Z,X],Y) + \xi(X,[Z,Y]) = 0,
\]
where we see $\xi$ as an element of $\Lambda^2(\frak g^*)$. Moreover, $\imath_Z\hat\phi =0$, so that
\[
c\,\omega(Z,X,Y) = - d\xi(Z,X,Y) = \xi([X,Y],Z).
\]
Hence
\[
c B([X,Y],Z) = \xi([X,Y],Z).
\]
As $\G$ is simple, we have $\frak g= [\frak g,\frak g]$ and therefore for every $U\in \frak g$
\beq\label{1}
\xi(U,Z) = cB(U,Z).
\eeq
Suppose now $\frak k\neq\{0\}$ and choose a non-zero element $Z\in\frak k$. By \eqref{1} we have that $0=c B(Z,Z)$, hence $c=0$ and $\imath_{\frak k}\xi=0$.
This means that $\xi$ descends to a $2$-form on $\M$ with $d\xi=\phi$. Consequently, every element in $H^3(M,\mathbb R)$ must be trivial, a contradiction.\par

We now prove \ref{b3b}). Using the same notation as in \ref{b3a}), suppose $0\neq [\phi]\in H^3(\M,\mathbb R)$ and consider the corresponding $3$-form $\hat\phi\in\Omega^3(\G)$.
The form $\hat\phi$ can be written as $\hat\phi = c_1\,\omega_1 + c_2\,\omega_2 + d\xi$, where $\omega_i$ is the standard $3$-form on $\G_i$, for $i=1,2$.
If $Z\in \gk$, then for every $X,Y\in\gg$
\[
c_1\,\omega_1(Z,X,Y)+c_2\,\omega_2(Z,X,Y) = - d\xi(Z,X,Y) = \xi([X,Y],Z).
\]
Hence
\[
\xi([X,Y],Z) = c_1 B_1([X,Y]_1,Z_1) + c_2 B_2([X,Y]_2,Z_2),
\]
where $B_i$ is the Cartan-Killing forms on $\frak g_i$, and the subscript ${-}_i$ denotes the projection along the $i^{th}$-component $\frak g_i$.
Now, suppose that, say, $p_1$	 has a non-trivial kernel containing some $Z\neq 0$ such that $Z_1=0$ and $Z_2\neq 0$.
As $\G$ is semisimple, we can express $Z = \sum_k[X_k,Y_k]$ for some vectors $X_k,Y_k\in\frak g$. Hence,
\[
0 = \xi(Z,Z) = c_2B_2(Z_2,Z_2),
\]
forcing $c_2=0$. If $c_1=0$, then $\hat\phi = d\xi$ and $\imath_{\gk}\xi = 0$, so that $\xi$ descends to an invariant $2$-form on $\M$ and $\phi$ is exact, a contradiction.
Therefore $c_1\neq 0$, and for every $Z'\in \gk$ we have $0=\xi(Z',Z') = c_1B_1(Z'_1,Z'_1)$ so that $Z'_1=0$, implying  $\gk\subseteq \frak g_2$.
\end{proof}

\begin{remark}
Note that if a projection, say $p_1$, is also surjective, namely $\gk_1=\gg_1$, then $\M$ is diffeomorphic to the simple factor $\G_2$ (up to a covering).
\end{remark}

\section{Compact homogeneous spaces with invariant BRF pairs}\label{5dimsect}

In this section, we investigate the existence of compact homogeneous spaces admitting invariant pairs $(g,H)$ such that
the corresponding Bismut connection $\nabla = \nabla^g + \frac12 g^{-1}H$ is Ricci flat and non-flat, and we aim at understanding whether the generalized Alekseevsky-Kimelfeld theorem may hold.

Basic examples of such spaces are given by compact (semi)simple Lie groups endowed with a bi-invariant metric $g$ and the standard harmonic $3$-form $\o$.
In these cases, the associated Bismut connection is flat.
Particular examples are given by the standard $3$-sphere $S^3\cong\SU(2)$ endowed with a constant curvature metric,
and the product $\M=S^3\times S^1$ with the product metric and the standard $3$-form on $S^3$ viewed as a $3$-form on $\M$ (cf.~\cite[Ex.~3.57]{FS}).

The $3$-dimensional case is settled in \cite[Prop.~3.55]{FS}, where it is proved that the existence of a BRF pair $(g,H)$ on a $3$-manifold $\M^3$ implies that $(\M^3,g)$ has constant sectional curvature 
and the associated Bismut connection is flat. 
The $4$-dimensional case can be completely understood following the reasoning in the proof of \cite[Thm.~8.26]{FS}. 
We specify it here for the reader's convenience.  
\footnote{We are indebted to Jeffrey Streets who pointed out to us this result.} 
\begin{proposition} 
Let $\M^4$ be a $4$-dimensional compact manifold admitting a BRF pair $(g,H)$. Then, the associated Bismut connection is flat. 
\end{proposition}
\begin{proof} 
As the $3$-form $H$ is harmonic, the $1$-form $\theta\coloneqq *H$ is also harmonic. Therefore, by Bochner formula, we have
\[
0=\int_M \langle \Delta \theta,\theta\rangle\, dV_g = \int_M \left({\Ric}(\theta^\#,\theta^\#) + ||\nabla^g\theta||^2\right)\, dV_g,
\]
so that $\theta$ is parallel and $0={\Ric}(\theta^\#,\theta^\#)=\frac 14 H^2(\theta^\#,\theta^\#)  = \frac14 ||\imath_{\theta^\sharp}H||^2$. 
This implies that the universal cover of $\M^4$ splits isometrically as a product $\mathrm{N}^3\times \mathbb R$, for some $3$-dimensional space $\mathrm{N}^3$, 
and that $\imath_vH=0$ whenever $v$ is tangent to the flat factor. The claim now follows from the $3$-dimensional case.  
\end{proof}

We now turn to the $5$-dimensional case and in the next theorem we determine which $5$-dimensional compact homogeneous spaces may admit invariant BRF pairs. 
\begin{theorem}\label{class5dim}
Let $(\M,g)$ be a compact $5$-dimensional homogeneous Riemannian manifold.
If $\M$ admits a harmonic $3$-form $H$ such that $(g,H)$ is a BRF pair, then one of the following holds:
\begin{enumerate}[i$)$]
\item $\M$ is finitely covered by a compact Lie group;
\item $\M$ is finitely covered by $\SU(2)^2/\T^1$, where $\T^1$ is embedded diagonally into $\SU(2)^2$.
\end{enumerate}
\end{theorem}
\begin{proof}
We consider the compact Lie group $\G$ given by the connected full isometry group of $(\M,g)$, and we recall that any harmonic form is invariant under the $\G$-action.
We start noting that we can assume $H\neq 0$, as otherwise by \cite{AK} the Ricci flat metric $g$ would be flat and $\M$ would be covered by a torus. \par

We first study the case where $\G$ is semisimple.
Let $\K$ denote the isotropy group of the $\G$-action on $\M$ at a fixed point $p$. Then, $\dim \K$ belongs to the set $\{1,2,3,4,5,6\}$.
Indeed, the isotropy representation embeds $\K$ into $\SO(5)$ and $\K$ is a proper subgroup, as the standard representation of $\SO(5)$ has no non-trivial invariant $3$-forms.
Moreover, a proper subgroup of $\SO(5)$ has dimension at most $6=\dim \SO(4)$, with equality if and only if $\K$ is conjugate to the standard $\SO(4)$.
Again, this case can be ruled out as there no non-trivial invariant $3$-forms.

The case $\dim \K=5$ cannot occur, as there are no $5$-dimensional subgroups of $\SO(5)$
(the rank is at most two: if the rank is $1$, then $\K\cong \T^1$ or $\SU(2)$, if the rank is $2$, then either $\K\cong \T^2$ or $\K\cong \T^1\cdot\SU(2)$, never with dimension $5$).\par

If $\dim \K=4$, then $\K$ and $\G$ are locally isomorphic to $\T^1\times \SU(2)$ and $\SU(2)^3$, respectively.
We denote by $\frak z$ the center of $\gk$ and by $\frak k_s\cong \su(2)$ the semisimple part of $\frak k$.
If $\gg=\oplus_{i=1}^3\gg_i$, with $\gg_i\cong \su(2)$, and $p_i:\gg\to\gg_i$ are the projections for $i=1,2,3$,
then we may suppose that $p_1(\gz)\neq\{0\}$, so that $p_1(\gk_s)=\{0\}$.
Since $\gk_s$ cannot coincide with $\gg_2$ or $\gg_3$ and since it is simple, we see that $p_i(\gk_s)=\gg_i$ and $p_i(\gz)=\{0\}$ for $i=2,3$.
Therefore, the manifold $\M$ is (up to a finite covering) $\G$-diffeomorphic to $\SU(2)/\T^1\times \SU(2)^2/\Delta\SU(2)\cong S^2\times S^3$.
In this case, the isotropy representation contains two inequivalent modules and therefore the above diffeomorphism maps the metric $g$ to a product metric.
As the $3$-form $H$ on $S^2\times S^3$ is the pull-back of a $3$-form on the $S^3$-factor,  we see that the Ricci tensor on the $S^2$-factor should vanish, a contradiction. \par

If $\dim \K = 3$, then $\dim \G= 8$ and, being semisimple, this implies $\G\cong \SU(3)$, which is simple. This contradicts Theorem \ref{b3}. 

The case $\dim \K = 2$ can be ruled out as there are no semisimple groups of dimension $7$.
The last case $\dim \K = 1$ forces $\G\cong \SU(2)^2$ and Theorem \ref{b3} says that $\K$ is embedded diagonally into $\G$,
unless $\gk$ is contained in one of the factors $\su(2)$. When this occurs, $\M$ is $\G$-diffeomorphic (up to a covering) to $\SU(2)\times S^2$.
As the isotropy $\gk$ acts on the tangent space of $S^2$ with no non-zero invariant vector and trivially on the tangent space of $\SU(2)$, we see that the above splitting is isometric.
Moreover, an invariant $3$-form $H$ on $\M$ is of the form $H_1+H_2$, where $H_1$ in an invariant form on $\SU(2)$, while $H_2 = \theta \wedge \sigma$, where $\theta$ is any invariant $1$-form on $\SU(2)$
and $\sigma$ is the volume form of $S^2$. As $dH_2=d\theta \wedge \sigma$ and $dH_1\in \Lambda^3(\su(2))$, we see that $dH=0$ forces $H=H_1$.
This means that $\imath_vH=0$ for every vector $v$ tangent to the $S^2$ factor, implying that $\Ric_{S^2}\equiv 0$, a contradiction.\par

We now deal with the non-semisimple case.
Let $\Z$ be the connected center of $\G$ and let $\L$ be the semisimple part of $\G$, which is a compact normal subgroup. We first summarize some basic observations:
\begin{enumerate}[(1)]
\item all $\L$-orbits are diffeomorphic. Indeed, $\L$ is a normal subgroup of $\G$, so for every $p\in \M$ and $g\in \G$ we have $\L\cdot gp=g(\L\cdot p)$. A generic $\L$-orbit will be denoted by $\mathcal O$;
\item Let $\U \coloneqq \{ z\in \Z \st z(\L\cdot p) = \L\cdot p,~\forall p\in \M\}$.
Then,  $\U$ is a closed subgroup of the torus $\Z$ and we can find a closed subgroup $\Z_1\subseteq \Z$ with $\Z_1\cap \U=\{e\}$,
$\Z = \U\cdot \Z_1$ and $\L\cdot \Z_1$ acting transitively on $\M$.
Indeed, at each point $q\in \M$ we have that $T_q\M = T_q(\Z_1\cdot q) \oplus T_q(\L\cdot q)$.
In detail, if $X\in \gz_1$ with $\hat X_q\in T_q(\L\cdot q)$, then for every $g\in \G$ we have $\hat X_{gq}= g_*\hat X_q\in T_{gq}(\L\cdot gq)$,
meaning that $X\in \gu$, whence $X=0$;
\item\label{rem3} the manifold $\M$ is $\G$-diffeomorphic (up to a finite covering) to the product $\Z_1\times \mathcal O$.
Indeed, the map $F : \Z_1\times \mathcal O\to \M$ given by $F(z,p)=z\cdot p$ is a local diffeomorphism, hence a covering thanks to the compactness of the involved spaces.
\end{enumerate}
Since $\L$ is non-trivial, its orbits have dimension at least $2$, whence $\dim \Z_1 = {\mathrm{chm}}(\M,\L)\leq 3$.
We now proceed by looking at the possible dimensions of $\Z_1$:
\begin{enumerate}[(a)]
\item $\dim \Z_1 = 1$. The generic $\L$-orbit $\mathcal O$ has dimension $4$ and, up to a finite covering, it is $\L$-diffeomorphic to $\SO(5)/\SO(4)$, $\SU(3)/\mathrm{S}(\U(1)\U(2))$ or $\SO(4)/\T^2$ (cf.~\cite{BB}).
As the isotropy representation has no non-trivial invariant vector in the tangent space of $\mathcal O$, we see that the splitting in (\ref{rem3}) is isometric.
As the space of invariant $3$-forms on $\mathcal O$ is trivial, the form $H$ is given by $H=dt\wedge\sigma$, where $dt\in\Lambda^1(\T^1)^{\T^1}$ and $\sigma$ is an $\L$-invariant $2$-form.
Then, $\Ric_g\left(\partial_t,\partial_t\right)=0$, while $||\imath_{\partial_t}H||^2 \neq 0$, a contradiction; \par
\item $\dim \Z_1=2$. In this case, $\dim \mathcal O=3$, and $\L$ must be either $\SU(2)$ or $\SO(3)$, up to covering.
In the former case, $\M$ is a Lie group up to covering; in the latter, $\mathcal O$ is covered by $S^3$ and the same argument as above shows that the splitting $\M=\Z_1\times S^3$ is isometric.
In this case, $H$ is a multiple of the volume form on the $S^3$ factor, as this has no non-trivial invariant $1$- or $2$-forms.
Moreover, the metric splits as the product of a flat metric on $\Z_1$ and a multiple of the standard metric on $S^3$.
Consequently, the associated Bismut connection is flat;
\item $\dim \Z_1=3$. Here $\mathcal O \cong S^2 = \SU(2)/\T^1$, so that the splitting (\ref{rem3}) is isometric.
The form $H$ can be expressed as $H=\phi\wedge\nu$, where $\phi$ is an invariant $1$-form on $\Z_1$ and $\nu$ is the volume form on $S^2$.
Again, $||\imath_vH||^2\neq 0$, for every $v$ in $T\Z_1$, while $\Ric_g(v,v)=0$, a contradiction.
\end{enumerate}\end{proof}

The previous result leads us to consider the homogeneous spaces $\M = (\SU(2)\times\SU(2) ) / \K$, with $\K\cong \T^1$ embedded diagonally.
Up to an automorphism of $\G = \SU(2)^2$, we can suppose that $\K$ is of the form
\[
\K_{p,q} \coloneqq \left\{(\mathrm{diag}(z^p,z^{-p}),\mathrm{diag}(z^q,z^{-q}))\in\SU(2)^2\ |\ z\in \T^1\right\},
\]
for some $p,q\in\mathbb N$, with $p\geq q\geq 1$ and $\rm{gcd}(p,q)=1$.
We then let $\M_{p,q} \coloneqq \G/\K_{p,q}$ and we recall that all these homogeneous spaces are diffeomorphic to $S^3\times S^2$, see \cite[Prop.~2.3]{WaZi},
although no explicit diffeomorphism is known (except when $p=q=1$).\par

\smallskip

We have the following.
\begin{theorem}\label{BismutMpq}
The $5$-dimensional compact homogeneous space $\M_{p,q}$ admits a $\G$-invariant BRF pair $(g_o,H_o)$ such that the associated Bismut connection is non-flat.
More precisely
\begin{enumerate}[$\bullet$]
\item when $p\neq q$, the space of $\G$-invariant BRF pairs $\mathcal B(\M_{p,q})^\G$ is given by $\mathbb R^{+}(g_o,H_o)$;
 \item when $p=q=1,$ the pair $(g_o,H_o)$ admits a suitable neighborhood $\mathcal{U}$ in $\mathcal M(\M_{1,1})\times \Omega^3(\M_{1,1})$ such that $\mathcal{U}\cap \mathcal B(\M_{1,1})^\G$ coincides with
 $\mathcal{U}\cap \R^+(g_o,H_o)$.
 \end{enumerate}
\end{theorem}

\begin{proof}
We begin observing that $\M_{p,q}$ does not admit any Bismut flat connection.
Indeed,  by \cite[Thm.~3.54]{FS}, a simply connected compact  Riemannian manifold admitting a Bismut flat connection is isometric to a product of compact simple Lie groups with bi-invariant metrics.
We obtain our claim by observing that there are no $5$-dimensional compact semisimple Lie groups.\par
In the Lie algebra $\su(2)$ we select the elements
\[
H=  \begin{pmatrix} \frac{i}{2}&0\\ 0&-\frac{i}{2}\end{pmatrix},\quad
E = \begin{pmatrix} 0&\frac 1{2\sqrt{2}}\\-\frac1{2\sqrt{2}}&0\end{pmatrix},\quad
V=  \begin{pmatrix} 0&\frac i{2\sqrt{2}}\\\frac i{2\sqrt{2}}&0\end{pmatrix},
\]
so that $[H,E]=V$, $[H,V]=-E$, $[H,V]= -E$ and $[E,V]=\frac 12 H$.
If $B$ denotes the Cartan-Killing form of $\su(2)$, then $B(E,E)=B(V,V)=-1$, $B(H,H)=-2$. Then, we can choose the following basis of $\gg$
\[
e_1 = (qH,-pH), \quad e_2 = (E,0),\quad e_3 = (V,0),\quad e_4 = (0,E), \quad e_5 = (0,V),\quad e_6 = (pH,qH),
\]
so that $\gk_{p,q}=\mathbb R e_6$, while $\gm\coloneqq {\mathrm{span}_\R}(e_1,\ldots,e_5)$ is an $\ad(\gk_{p,q})$-invariant subspace
and $\frak g= \gk_{p,q}+\gm$ is an $\ad(\gk_{p,q})$-invariant $B$-orthogonal decomposition.
An $\ad(\gk_{p,q})$-irreducible decomposition of $\gm$ is given by
\[
\gm = \gm_0 \oplus \gm_1 \oplus \gm_2,
\]
where $\gm_0=\R  e_1$, $\gm_1= \mathrm{span}_\R(e_2,e_3)$ and  $\gm_2=\mathrm{span}_\R(e_4,e_5)$. Moreover, we have
\[
\ad(e_6)e_1 = 0,\quad \ad(e_6)e_2 = p\,e_3,\quad \ad(e_6)e_3 = -p\,e_2,\quad \ad(e_6)e_4 = q\,e_5,\quad \ad(e_6)e_5 = -q\,e_4,
\]
thus the module $\gm_0$ is trivial, while the modules $\gm_1$ and $\gm_2$ are non-trivial and inequivalent if $p\neq q$, and non-trivial and equivalent if $p=q$.
We will discuss the cases $p\neq q$ and $p=q$ separately.

In the following,  $\mathcal{B}^* = (e^1,e^2,e^3,e^4,e^5)$ denotes the dual basis of $\mathcal{B}$, and the shortening $e^{ijk\cdots}$ is used to denote the wedge product of covectors $e^i\W e^j \W e^k \W \cdots$.
Moreover, we fix the orientation on $\gm$ for which $\mathcal{B}$ is positively oriented.

\smallskip \noindent
{\bf Case ${p\neq q}$.}
We have
\[
(\Lambda^3\gm^*)^{\gk_{p,q}} = \gm_0^*\otimes \Lambda^2\gm_1^* \oplus \gm_0^*\otimes \Lambda^2\gm_2^*,
\]
and a generic invariant 3-form is given by
\[
H = h_1 e^{123} + h_2 e^{145},
\]
for some $h_1,h_2\in\mathbb{R}$. Using the Koszul formula for the differential of invariant forms on $\gm$
\begin{equation}\label{Koszul}
dH(X_0,X_1,X_2, X_3) = \sum_{i<j} (-1)^{i+j}\ H \left([X_i,X_j]_{\gm},X_k,X_l\right),\quad X_0,\ldots, X_3\in \gm,
\end{equation}
we see that $H$ is closed if and only if $(h_1,h_2) = (\lambda q,\lambda p)$, for some $\lambda\in\R$.
Now, the generic invariant metric on $\gm$ has the following expression
\[
g = \mu^2 e^1\otimes e^1 + a^2 \left(e^2\otimes e^2 + e^3 \otimes e^3 \right) + b^2 \left(e^4\otimes e^4 + e^5 \otimes e^5 \right),
\]
for some positive real numbers $\mu,a,b$. We can then easily compute the Hodge dual of $H$ obtaining
\[
*_gH = \frac{a^2}{\mu\,b^2}\,h_2\,e^{23} + \frac{b^2}{\mu\,a^2}\,h_1\,e^{45},
\]
and we see that it is always closed.
Thus, up to a constant multiple, the generic invariant harmonic 3-form on $\gm$ is given by
\[
H = q\,e^{123} + p\, e^{145}.
\]
Now, we compute the symmetric $2$-tensor $H^2$ and we see that its non-zero components are the following
\[
\begin{split}
H^2(e_1,e_1)	&= 2\, \frac{a^4p^2+b^4q^2}{a^4b^4},\\
H^2(e_2,e_2)	&= 2\, \frac{q^2}{a^2\mu^2} = H^2(e_3,e_3), \\
H^2(e_4,e_4)	&= 2\, \frac{p^2}{b^2\mu^2} = H^2(e_5,e_5).
\end{split}
\]
As for the Ricci tensor $\Ric_g$, we choose a $g$-orthonormal basis $(E_1,E_2,E_3,E_4,E_5)$ of $\gm$, and we compute its components with respect to the basis $\mathcal{B}$ using the formula \cite[7.38]{Bes}.
Since $\gg$ is unimodular, this formula reads
\begin{equation}\label{RicHom}
\Ric_g(X,X) = -\frac12 \sum_{i=1}^5 ||[X,E_i]_{\gm}||^2 -\frac12 B(X,X) +\frac12 \sum_{1\leq i < j \leq 5} g([E_i,E_j]_{\gm},X)^2,
\end{equation}
for every $X\in\gm$.
We obtain that the only non-zero components of $\Ric_g$ are the following
\[
\begin{split}
\Ric_g(e_1,e_1)	&= \mu^4\frac{a^4p^2 + b^4q^2}{8 a^4 b^4 \left(p^2 + q^2\right)^2}, \\
\Ric_g(e_2,e_2)	&= \frac{4a^2\left(p^2 + q^2\right)^2 - \mu^2 q^2}{8a^2 \left(p^2 + q^2\right)^2}  =  \Ric_g(e_3,e_3), \\
\Ric_g(e_4,e_4)	&= \frac{4b^2\left(p^2 + q^2\right)^2 - \mu^2 p^2}{8b^2 \left(p^2 + q^2\right)^2}  =  \Ric_g(e_5,e_5).
\end{split}
\]
Now, it is easy to see that $\Ric_g = \frac14 H^2$ has a unique solution under the constraints $\mu>0,a>0,b>0$ and $p^2+q^2\neq0$, namely
\[
\mu = \sqrt{2\left(p^2+q^2\right)},\quad a = \sqrt{\frac{q^2}{p^2+q^2}},\quad b = \sqrt{\frac{p^2}{p^2+q^2}}.
\]
Thus, when $p\neq q$, the homogeneous space $\M_{p,q}=\G/\K_{p,q}$ admits an invariant metric $g_o$ and an invariant harmonic form $H_o$ giving rise to a Bismut Ricci flat connection.
Moreover, the pair $(g_o,H_o)$ is unique up to scaling.

\smallskip \noindent
{\bf Case ${p= q}$.}
We have $p=q=1$ and the modules $\gm_1$ and $\gm_2$ are equivalent. The space of $\ad(\gk_{1,1})$-invariant 3-forms is given by
\[
(\Lambda^3\gm^*)^{\gk_{1,1}} = \gm_0^*\otimes \Lambda^2\gm_1^*   \oplus \gm_0^*\otimes \Lambda^2\gm_2^* \oplus \left(\gm_0^* \otimes \gm_1^*\otimes \gm_2^*\right)^{\gk_{1,1}},
\]
with $\dim \left(\gm_0^* \otimes \gm_1^*\otimes \gm_2^*\right)^{\gk_{1,1}} = 2$.
The generic invariant 3-form $H\in(\Lambda^3\gm^*)^{\gk_{1,1}}$ has the following expression
\beq\label{invH}
H = h_1\,e^{123}  + h_2\,e^{145} + h_3 \left(e^{125}-e^{134}\right) + h_4\left(e^{124} + e^{135} \right),
\eeq
where $h_i\in\R$, for $i=1,2,3,4$. Using the Koszul formula \eqref{Koszul}, we see that $H$ is closed if and only if $h_2=h_1$.

The generic invariant metric on $\gm$ is given by
\beq\label{invg}
\begin{split}
g 	&= \mu^2 e^1\odot e^1 + a^2 \left(e^2\odot e^2 + e^3 \odot e^3 \right) + b^2 \left(e^4\odot e^4 + e^5 \odot e^5 \right) \\
	&\quad + 2c\left(e^2\odot e^4 + e^3 \odot e^5\right) +2s\left(e^2\odot e^5 - e^3\odot e^4\right),
\end{split}
\eeq
where $\mu, a, b, c, s$ are real constants such that
\[
\mu>0,\quad a>0,\quad b>0,\quad a^2b^2-c^2 + s^2>0,
\]
and $e^i\odot e^j \coloneqq \frac12\left(e^i\otimes e^j + e^j \otimes e^i\right)$.
We remark here that it is always possible to assume $s=0$.
Indeed, the one-dimensional torus $\U\coloneqq \exp(\mathbb R e_1)$ centralizes $\K_{1,1}$ and therefore for every $u\in\U$ we can consider the $\G$-equivariant diffeomorphism $\tau_u\in {\rm{Dif{}f}}(\M)^\G$ given by
$\tau_u(x\K_{1,1}) = xu\K_{1,1}$, for $x\in \G$. Then, for every $g\in S^2(\gm)^{\gk_{1,1}}$, which is given by the data $(\mu,a,b,c,s)$ as in \eqref{invg},
the $\G$-invariant symmetric tensor $\tau_u^*g$ corresponds to the data $(\mu,a,b,c'=c\,\cos t + s\, \sin t,s'= -c\, \sin t + s\, \cos t)$, for $u=\exp(te_1)\in\U$.
Therefore, we immediately see that we have $s'=0$ for an appropriate choice of $u\in\U$.\par

Now, we have
\[
\begin{split}
*_gH &= \frac{h_1\left(a^4+ c^2\right) - 2a^2c h_3 }{\mu\left(a^2b^2 - c^2\right)}\, e^{23}  + \frac{ h_1\left(b^4 + c^2\right) - 2b^2ch_3}{\mu\left(a^2b^2 - c^2\right)}\, e^{45}\\
	&\quad -\frac{h_4}{\mu}\, \left(e^{24}+e^{35}\right) - \frac{h_3\left( a^2 b^2 + c^2\right)  -c h_1\left( a^2 + b^2\right)  }{\mu\left(a^2b^2 - c^2\right)} \left(e^{25}-e^{34}\right),
\end{split}
\]
and using again the Koszul formula, we obtain
\[
d*_gH = 2\,\frac{h_4}{\mu}\left(e^{125}  - e^{134}\right) -2\,\frac{h_3\left(a^2 b^2 +c^2\right)   -c h_1\left( a^2 + b^2\right)  }{\mu\left(a^2b^2 - c^2\right)}\left(e^{124} + e^{135} \right).
\]
Thus, $H$ is coclosed if and only if
\[
h_3 =   \frac{c h_1\left( a^2 + b^2\right)}{a^2 b^2 +c^2 }, \quad h_4 = 0.
\]
The generic invariant harmonic 3-form on $\gm$ is then given by
\beq\label{invhH}
H = h_1\left( e^{123}  + e^{145} +  c\,\frac{  a^2 + b^2 }{a^2 b^2 +c^2 } \left(e^{125}-e^{134}\right) \right).
\eeq
Consequently, the (possibly) non-zero components of the invariant symmetric $2$-tensor $H^2$ are the following
\[
\begin{split}
H^2(e_1,e_1)	&= 2 h_1^2\, \frac{a^4 + b^4 - 2c^2}{\left(a^2b^2 - c^2\right)\left(a^2b^2 + c^2\right)},\\
H^2(e_2,e_2)	&= 2 h_1^2\, \frac{a^2c^2 \left(a^4 + b^4\right)  -c^4 \left(2a^2 +  b^2\right) + a^4b^6}{\mu^2\left(a^2b^2 - c^2\right)\left(a^2b^2 + c^2\right)^2} = H^2(e_3,e_3),\\
H^2(e_4,e_4)	&= 2 h_1^2\, \frac{b^2 c^2 \left(a^4 + b^4\right) -c^4\left(a^2 + 2b^2\right)  + a^6b^4}{\mu^2\left(a^2b^2 - c^2\right)\left(a^2b^2 + c^2\right)^2} = H^2(e_5,e_5), \\
H^2(e_2,e_4)	&= 2 c h_1^2\, \frac{ a^2b^2\left(  a^4 + a^2 b^2 +  b^4 - 2 c^2\right) - c^4}{\mu^2\left(a^2b^2 - c^2\right)\left(a^2b^2 + c^2	\right)^2} = H^2(e_3,e_5).
\end{split}
\]
Let us consider the following $g$-orthonormal basis of $\gm$
\[
\begin{split}
E_1 	&= \frac{1}{\mu}\,e_1,\quad E_i	= \frac{1}{a}\,e_i,~i=2,3,\\
E_j	&= -\frac{c}{a\sqrt{a^2b^2-c^2}}\,e_{j-2} + \frac{a}{\sqrt{a^2b^2-c^2}}\,e_j,~j=4,5.
\end{split}
\]
Then, using formula \eqref{RicHom} and its polarization, we obtain the following expressions for the (possibly) non-zero components of the Ricci tensor
\[
\begin{split}
\Ric_g(e_1,e_1)	&= \frac{ 2c^2\left( 64 c^2 - 64 a^2 b^2 - \mu^4\right) + \mu^4 (a^4 + b^4)}{32 \left(a^2 b^2 - c^2\right)^2},  \\
\Ric_g(e_2,e_2)	&= \frac{ 64a^2c^2 + \mu^2\left(16a^2b^2 -b^2\mu^2  - 16c^2\right) }{32 \mu^2 \left(a^2b^2 - c^2\right)} = \Ric_g(e_3,e_3),\\
\Ric_g(e_4,e_4)	&= \frac{64b^2c^2 +\mu^2 \left( 16a^2b^2 -a^2\mu^2 - 16c^2\right)}{32\mu^2\left(a^2b^2 - c^2\right)} = \Ric_g(e_5,e_5), \\
\Ric_g(e_2,e_4)	&= c\, \frac{64a^2b^2 - \mu^4}{32\mu^2\left(a^2b^2 - c^2\right)}= \Ric_g(e_3,e_5).
\end{split}
\]
We now look for invariant metrics $g$ and harmonic 3-forms $H$ for which $\Ric_g = \frac14 H^2$.

A computation using the above expressions of the components of $\Ric_g$ and $H^2$ shows that
we need to find points in the open subset
\[
\mathcal{A} \coloneqq \left\{(\mu,a,b,c,h_1) \in \R^5 \st \mu>0, a>0, b>0, a^2b^2-c^2>0, h_1\neq0 \right\} \subset \R^5,
\]
where all of the following polynomials vanish
\[
\begin{split}
p_1	&= \left(a^{4}+b^{4}-2 c^{2}\right) \left[\mu^{4} \left(a^{2} b^{2}+c^{2}\right)-16\,h_{1}^{2} \left(a^2b^2 -c^2 \right) \right] -128\, c^{2}  \left(a^4b^4 -c^4 \right), \\
p_2	&= \left(a^{2} b^{2}+c^{2}\right)^{2} \left[16\,\mu^{2} \left(a^{2} b^{2}-c^{2}\right) +64 a^{2} c^{2} -b^{2} \mu^{4}\right]-16 h_{1}^{2} \left[a^{4} b^{6}+ c^{2}\left(a^{6} +a^{2} b^{4} -2 a^{2} c^{2}-b^{2} c^{2}\right)\right], \\
p_3	&= \left(a^{2} b^{2}+c^{2}\right)^{2} \left[16\,\mu^{2} \left(a^{2} b^{2}-c^{2}\right) +64 b^{2} c^{2} -a^{2} \mu^{4}\right]-16 h_{1}^{2} \left[a^{6} b^{4}+ c^{2} \left(b^{6}  +a^{4} b^{2} -2 b^{2} c^{2} -a^{2} c^{2} \right)\right], \\
p_4	&= c \left[\left(a^{2} b^{2}+c^{2}\right)^{2} \left(64 a^{2} b^{2}-\mu^{4}\right)-16 h_{1}^{2} \left(a^2b^2 \left(a^{2} + b^{2}\right)^{2}-\left(a^{2} b^{2}+c^{2}\right)^{2}\right)\right].
\end{split}
\]
Clearly, the polynomial $p_4$ vanishes if $c=0$. When this happens, the remaining polynomials simplify considerably, and they vanish simultaneously if and only if
\[
\mu= 2\sqrt{2}\,t,\quad a = b = t,\quad h_1 = \pm2t^2,
\]
for some $t>0$.
Therefore, when $c=0$, we obtain an invariant BRF pair $(g_o,H_o) \in \mathcal{B}(\M_{1,1})^\G$ that is unique up to scaling and up to a sign in the definition of $H_o$, namely
\[
g_o = 8\, e^1\odot e^1 +  \left(e^2\odot e^2 + e^3 \odot e^3 \right) + \left(e^4\odot e^4 + e^5 \odot e^5 \right), \quad
H_o = 2e^{123}  + 2e^{145}.
\]
To prove the last assertion, we first consider the set
\[
\mathcal P \coloneqq \left\{(g,H)\in S^2(\gm)^{\gk_{1,1}}\times \Lambda^3(\gm)^{\gk_{1,1}} \st g >0,~dH=0,~\d_gH=0\right\},
\]
and we recall that every invariant metric is uniquely determined by the string $(\mu,a,b,c,s)$ as in \eqref{invg}. We then consider the subset
\[
\mathcal P' \coloneqq \left\{(g,H)\in \mathcal P|\ g(e_2,e_5)=0\right\}.
\]
Every point $(g,H)\in\mathcal P'$ is uniquely determined by a string $(\mu,a,b,c,h_1)\in \mathcal{A} \subset  \mathbb R^5$, where $(\mu,a,b,c)$ determines $g$ as in \eqref{invg}
and $h_1$ determines $H$ as in \eqref{invhH}.
The set $\mathcal B(\M_{1,1})^{\G} \cap \mathcal P'$ can be identified with the zero set in $\mathcal A$ of the map
\[
F : \mathcal{A} \rightarrow \R^4,\quad F(\mu,a,b,c,h_1) = (p_1,p_2,p_3,p_4).
\]
Notice that $F(x_o)=0$, where $x_o=(2\sqrt{2},1,1,0,2)\in\mathcal{A}$ corresponds to the BRF pair $(g_o,H_o)$, and that $F(\g(t))=0\in\mathbb R^4$ for every $t\in \R^+$,
where $\g(t) = (2\sqrt{2}t,t,t,0,2t^2)$ is the curve through $x_o$ corresponding to $(t^2g_o,t^2H_o)$. If we now compute the differential of $F$ at $x_o$, we obtain
\[
F_{*x_o} =
\left(
\begin{array}{ccccc}
128 \sqrt{2} & 0 & 0 & 0 & -128 \\
 0 & 256 & 0 & 0 & -64 \\
 0 & 0 & 256 & 0 & -64 \\
 0 & 0 & 0 & -192 & 0
\end{array}
\right),
\]
and since $\mathrm{rank} \left(F_{*x_o} \right) = 4,$ we see that there exists a neighborhood $\mathcal U'$ of $(g_o,H_o)$ in $\mathcal P'$ so that $\mathcal U' \cap \mathcal B(\M_{1,1})^{\G}$
coincides with the set $\left\{(t^2g_o,t^2H_o) \st t\in (1-\varepsilon,1+\varepsilon')\right\}$, for suitable $\varepsilon,\varepsilon' >0$.\par
We now use the $\G$-equivariant diffeomorphisms $\tau_u$, $u\in \U$, to transform every element of $\mathcal B(\M_{1,1})^{\G}\cap \mathcal P$ into an element of
$\mathcal B(\M_{1,1})^{\G}\cap \mathcal P'$ via $\tau_u^*$.
Note that for every $u\in \U$ we have $(\tau_u^*g_o,\tau_u^*H_o)=(g_o,H_o)$, since $c=s=0$. Moreover, we can easily see that there exists a
neighborhood $\mathcal U$ of $(g_o,H_o)$ in $\mathcal P$ so that $\tau_u^*(\mathcal U)\subseteq \mathcal U'$.
Therefore, if $(g,H)\in \mathcal B(\M_{1,1})^{\G}\cap \mathcal U$, then we can find $u\in \U$ so that $\tau_u^*(g,H)\in \mathcal B(\M_{1,1})^{\G}\cap \mathcal U'$.
Hence, $\tau_u^*(g,H)= (t^2g_o,t^2H_o)$, for some $t > 0$, and we have $(g,H)=\tau_{u^{-1}}^*(t^2g_o,t^2H_o) = (t^2g_o,t^2H_o)$.
\end{proof}

\begin{remark}
The Bismut Ricci flat, non-flat manifolds constructed in the previous theorem provide counterexamples to a generalized Alekseevsky-Kimelfeld theorem \cite{AK} for the Bismut connection. 
This question was raised in \cite{FS}, cf.~Question 3.58. We also remark here that the Riemannian manifolds $(\M_{p,q},g_o)$  are not Einstein.
\end{remark}

\begin{remark}
In the proof, we have noticed that the BRF pairs $(g,H)$ are unique (up to multiples) when $p\neq q$.
When $p=q=1$, we have given the full description of all possible invariant metrics and relative harmonic $3$-forms.
Unfortunately, the complexity of the computations prevented us from finding other solutions.
\end{remark}

\begin{remark}
We remark that the $3$-form $H$ we have constructed in the Riemannian spaces $(\M_{p,q},g)$ is harmonic but never parallel with respect to the Bismut connection.
Indeed, if it were, then we would have $dH = 2\sigma_H = 0$, where $\sigma_H$ is the {\em fundamental $4$-form} $\sigma_H=\frac 12 \sum_{i=1}^5\imath_{v_i}H\wedge\imath_{v_i}H$ and $\{v_i\}$ is a local orthonormal frame
(see e.g.~\cite{FI}). By \cite[Thm.~4.1]{AFF}, we would then have that $\M_{p,q}$ is a compact simple Lie group, a contradiction.
\end{remark}

\subsection{The Kobayashi's construction}\label{KobaConst}
The manifolds $\M_{p,q}$ naturally appear as principal $S^1$-bundles over $\SU(2)^2/\T^2\cong S^2\times S^2$.
In this section, we briefly review a useful construction, due to Kobayashi \cite{K}, which allows to obtain new geometric structures on principal $S^1$-bundles over some base manifold $B$.
This might lead to a generalization of our examples. \par
Given a compact manifold $B$, it is well known that there is a one-to-one correspondence between principal $S^1$-bundles $\pi:P\to B$ and elements in $H^2(B,\mathbb Z)$.
Given a closed $2$-form $\a$ with $[\a]\in {H}^2(B,\mathbb Z)$, there exists a principal $S^1$-bundle $\pi:P\to B$ and a connection form $\g$ so that $d\g=\pi^*\a$.
If $B$ is equipped with a Riemannian metric $g_o$, we can construct a Riemannian metric $g$ on $P$ by setting $g \coloneqq c^2\g\odot\g + \pi^*g_o$, for some $c>0$.
Following \cite{K}, we compute the Ricci curvature as follows: if $X,Y$ are horizontal tangent vectors and $v$ is a unit vertical tangent vector, we have
\[
\begin{split}
\Ric_g(X,Y)	&= \Ric_{g_o}(\pi_*X,\pi_*Y) - 2 \hat\a(\pi_*X,\pi_*Y),\\
\Ric_g(X,v) 	&= c\, \delta_{g_o}\a(\pi_*X),\\
\Ric_g(v,v) 	&= c^2\, ||\a||^2,
\end{split}
\]
where for every $2$-form $\omega$ we define $\hat\o\in S^2(T^*B)$ as $\hat\o(Z,W) = g_o(\imath_Z\o,\imath_W\o)$, for every tangent vector fields $Z,W\in\Gamma(TB)$, and we let $||\o||^2= g_o(\hat\o,\hat\o)$. \par
We now prove a result that might be useful to construct new examples of BRF pairs on suitable manifolds.
In particular, this result allows to reduce the problem of finding harmonic $3$-forms on a certain manifold to a problem on the existence of suitable harmonic $2$-forms on a lower dimensional manifold.

\begin{proposition}\label{KobaProp}
Given a compact Riemannian manifold $(B,g_o)$ and a non-zero harmonic $2$-form $\a$ with $[\a]\in {H}^2(B,\mathbb Z)$,
the associated principal $S^1$-bundle $P$ over $B$ admits a BRF pair $(g,H)$ if there exist a harmonic form $\beta\in \Omega^2(B)$ and positive numbers $\lambda,\mu\in \mathbb R^+$
so that the following conditions are fulfilled:
\begin{enumerate}[a$)$]
\item\label{Koa} $\a\wedge\b=0$;
\item\label{Kob} $||\b||^2 = \lambda ||\a||^2$ at every point of $B$;
\item\label{Koc} the Ricci tensor $\Ric_{g_o}$ satisfies the equation
\[
\Ric_{g_o} = 2 \hat \a + \mu\hat \b.
\]
\end{enumerate}
\end{proposition}
\begin{proof}
We consider the metric $g= c^2\g\odot\g + \pi^*g_o$, for some $c>0$, as above. We then define a $3$-form $H\in\Omega^3(P)$ as
\[
H \coloneqq h\, \g\wedge\pi^*\b,
\]
for some $h\in\mathbb R$.
Searching for conditions implying that $(g,H)$ is a BRF pair, we see that $H$ is closed if and only if $\alpha\wedge\beta=0$ and $\beta$ is closed, while 
it is coclosed if and only if $ d(\pi^*(*_{g_o}\b))=0$, i.e., $\b$ is coclosed.
Therefore, $H$ is $g$-harmonic if and only if $\beta$ is $g_o$-harmonic and condition \ref{Koa}) holds.
Looking now at the expression of the Ricci tensor of $g$, we see that $\a$ being harmonic implies that $\Ric_g(V,Z)$ vanishes whenever $V$ is vertical and $Z$ is horizontal.
As we easily check,  $H^2(Z,V)=0$ as well. Thus, we only need to verify the following
\beq \label{1}
\Ric_{g_o}(X,Y) = 2   \hat \a(X,Y) + \frac 14 H^2(X,Y),
\eeq
\beq \label{2}
c^2||\a||^2 = \frac 14 ||\imath_vH||^2,
\eeq
where $X,Y$ are vector fields on $B$ and $v$ is a unit vertical field on $P$. Now, using \ref{Kob}), we have
\[
||\imath_vH||^2 = h^2 (\g(v))^2 ||\b||^2 = \frac {h^2}{c^2}||\b||^2 = \lambda \frac {h^2}{c^2}||\a||^2,
\]
hence equation \eqref{2} reads
\[
c^4 = \frac 14\lambda h^2.
\]
As for equation \eqref{1}, we see that $H^2(X,Y)=\frac{h^2}{c^2} \hat\b(X,Y)$ so that \eqref{1} reads
\[
\Ric_{g_o} = 2\hat\a + \frac{h^2}{4c^2}\hat\b.
\]
Comparing this last identity with \ref{Koc}), we see that the constants $c$ and $h$ must satisfy
\[
\begin{cases}
4c^4 = \lambda h^2,\\
4c^2 = \frac{h^2}{\mu}.
\end{cases}
\]
Choosing $c^2 = \lambda\mu$ and $h^2 = 4\lambda\mu^2$, we have that $(g,H)$ defines a BRF pair on $P.$
\end{proof}

\begin{remark}
Our examples on the manifolds $\M_{p,q}$ can be seen as special cases of the above construction.
\end{remark}

\section{The asymptotic behavior of the solution to the GRF on $\M_{p,q}$ }\label{Mpqasymp}

In this last section, we consider the homogeneous space $\M_{p,q} = \G/\K_{p,q}$, $p\neq q$, with the BRF pair $(g_o,H_o)$ determined in the proof of Theorem \ref{BismutMpq},
and we study the behavior of the generalized Ricci flow \eqref{GRF} starting at an invariant pair $(g,H)$ with $dH=0$.

We look for invariant solutions $(g_t,H_t)$ to the flow.
It follows from the discussion in the proof of Theorem \ref{BismutMpq} that we have
\[
g_t = \mu_t^2 e^1\otimes e^1 + a_t^2 \left(e^2\otimes e^2 + e^3 \otimes e^3 \right) + b_t^2 \left(e^4\otimes e^4 + e^5 \otimes e^5 \right),
\]
and
\[
H_t = \lambda_t H_o =  \lambda_t\left( q\,e^{123} + p\, e^{145}\right),
\]
where $\mu_t,a_t,b_t,\lambda_t$ are (positive) real valued functions of $t$.

The BRF pair $(g_o,H_o)$ is a fixed point of the generalized Ricci flow, thus the solution to the flow starting from it corresponds to the constant functions
\[
\mu_o \equiv \sqrt{2\left(p^2+q^2\right)},\quad a_o \equiv \sqrt{\frac{q^2}{p^2+q^2}},\quad b_o \equiv \sqrt{\frac{p^2}{p^2+q^2}},\quad \lambda_o \equiv 1.
\]

Now, since $H_t$ is $g_t$-harmonic, from the second equation in \eqref{GRF} we obtain
\[
0 = -\Delta_{g_t}H_t = \frac{\partial}{\partial t} H_t = \frac{d}{dt}\lambda_t H_o,
\]
whence it follows that $\lambda_t =\lambda\in\R^+$ is a positive constant and thus $H_t  =\lambda H_o$.
We now consider the first equation in \eqref{GRF}
\[
 \frac{\partial}{\partial t} g_t = -2\,\Ric_{g_t}+\frac12 H^2_t =  -2\, \Ric_{g_t}+\frac12 \lambda^2H^2_o.
\]
Using again our previous computations, we see that this equation is equivalent to an autonomous system of ODEs for the functions $\mu_t,a_t,b_t$.
If we let  $M_t\coloneqq \mu_t^2$, $A_t\coloneqq a_t^2$, $B_t\coloneqq b_t^2$, then the system is the following:
\begin{equation}\label{odeMpq}
\left\{
\begin{split}
\frac{d}{dt} M_t		&= \left(\frac{p^2}{4 \left(p^{2}+q^{2}\right)^{2}} \frac{1}{B_t^{2}} + \frac{q^2}{4  \left(p^{2}+q^{2}\right)^{2} } \frac{1}{A_t^{2}}   \right)
					\left(4\, \lambda^{2} \left(p^{2}+q^{2}\right)^{2}-M_t^{2}\right),  \\
\frac{d}{dt} A_t		&= \frac{q^2}{A_t} \left( \frac{ \lambda^{2} }{  M_t } +  \frac{1}{4  \left(p^{2}+q^{2}\right)^{2} } M_t\right) - 1,\\
\frac{d}{dt} B_t		&= \frac{p^2}{B_t} \left( \frac{ \lambda^{2} }{ M_t } +  \frac{1}{4  \left(p^{2}+q^{2}\right)^{2} } M_t \right) - 1.
\end{split}
\right.
\end{equation}
The fixed point of this system is given by $ x_{o,\lambda} \coloneqq \left(\lambda \mu_o^2,  \lambda a_o^2, \lambda b_o^2\right)$ and it corresponds to the invariant metric $\lambda g_o$.
We can write the system \eqref{odeMpq} as follows
\[
\frac{d}{dt}  (M_t,A_t,A_t) = f(M_t,A_t,B_t),
\]
and computing the Jacobian of $f$ at the fixed point $x_{o,\lambda}$, we obtain
\[
\mathrm{J}(f)_{x_{o,\lambda}} = \begin{pmatrix}
- \frac{\left(p^2+q^2\right)^2}{\lambda p^2q^2}	&	0	&	0 \\
0	&	- {\frac{p^2+q^2}{\lambda q^2}}		&	0\\
0	&	0	&	-{\frac{p^2+q^2}{\lambda p^2}}
\end{pmatrix}.
\]
Since $\mathrm{J}(f)_{x_{o,\lambda}}$ has negative real eigenvalues, we see that the fixed point $x_{o,\lambda}$ of \eqref{odeMpq} is asymptotically stable, namely
every solution $x_t = (M_t,A_t,B_t)$ to \eqref{odeMpq} starting sufficiently close to $x_{o,\lambda}$ exists for all positive time,
it stays close to $x_{o,\lambda}$, and their distance tends to zero as $t\rightarrow +\infty$.
In other words, the BRF pair $(\lambda g_o,\lambda H_o)$ on $\M_{p,q}$ is asymptotically stable along the generalized Ricci flow. 

A direct qualitative analysis of the system \eqref{odeMpq} shows that the solution $(M_t,A_t,B_t)$ starting at any given triplet of positive numbers  $(M_o,A_o,B_o)$ at $t=t_o$
exists for all $t \geq t_o$ and it converges to the fixed point $x_{o,\lambda}$ as $t\rightarrow+\infty$.

\begin{remark}
The analogous study of the behavior of the generalized Ricci flow on $\M_{1,1}$ is much more involved, due to the presence of off-diagonal terms in the invariant symmetric tensors.
\end{remark}

\noindent
{\bf Acknowledgements.}
The authors were supported by GNSAGA of INdAM and by the project PRIN 2017  ``Real and Complex Manifolds: Topology, Geometry and Holomorphic Dynamics''.
The authors warmly thank Jeffrey Streets for useful comments on a preliminary version of this work.

\end{document}